\documentclass[12pt]{amsart}
\usepackage{a4wide,enumerate,xcolor}
\usepackage{amsmath,graphicx,comment}
\allowdisplaybreaks

\let\pa\partial
\let\na\nabla
\let\eps\varepsilon
\newcommand{\N}{{\mathbb N}}
\newcommand{\R}{{\mathbb R}}

\newcommand{\dd}{{\rm d}}

\newtheorem{theorem}{Theorem}
\newtheorem{lemma}[theorem]{Lemma}
\newtheorem{proposition}[theorem]{Proposition}

\begin{document}

\title[Chemotaxis guidance of random walkers]{Chemotaxis guidance of random walkers \\ modeling self-wiring of neural networks} 

\author[N. Geltner]{Noah Geltner}
\address{Institute of Analysis and Scientific Computing, TU Wien, Wiedner Hauptstra\ss e 8--10, 1040 Wien, Austria}
\email{noah.geltner@tuwien.ac.at} 

\author[A. J\"ungel]{Ansgar J\"ungel}
\address{Institute of Analysis and Scientific Computing, TU Wien, Wiedner Hauptstra\ss e 8--10, 1040 Wien, Austria}
\email{juengel@tuwien.ac.at} 

\date{\today}

\thanks{The authors acknowledge partial support from   
the Austrian Science Fund (FWF), grant 10.55776/F65, and from the Austrian Federal Ministry for Women, Science and Research and implemented by \"OAD, project MultHeFlo. This work has received funding from the European Research Council (ERC) under the European Union's Horizon 2020 research and innovation programme, ERC Advanced Grant NEUROMORPH, no.~101018153. For open-access purposes, the authors have applied a CC BY public copyright license to any author-accepted manuscript version arising from this submission.} 

\begin{abstract}
A stochastic walker model is proposed to describe the chemotactic guidance of growth cones, i.e.\ the tips of developing neurites. The model accounts for the influence of both attractive and repulsive chemical cues, which are emitted by the growth cones and the somas. The system couples stochastic differential equations governing the motion of the growth cones with reaction--diffusion equations that describe the dynamics of the chemical concentrations. The existence of a unique solution to this coupled system is proved. Numerical experiments are performed to investigate the sensitivity of the model to key biological parameters. The impact of the nonlocal regularization of point sources in the reaction--diffusion equations is analyzed in a simplified deterministic setting.
\end{abstract}

 

\maketitle



\section{Introduction}

The formation of neural networks is a highly involved process governed by many biochemical markers. Roughly speaking, the main body cell of a neuron (soma) releases neurites (cellular extensions), which follow complex external and internal cues to form synaptic connections with other neurites or somas. The tip of a growing neurite has a dynamic structure called a growth cone, which guides the neurite to its target location \cite{NaSp17}. We focus on external cues given by chemical substances that are emitted by the growth cones and the somas. Following \cite{SeBe00}, the growth cones are modeled as random walkers driven by the gradient of chemical signals, which are produced by the random walkers via reaction--diffusion equations. For details and more complex neuronal growth models, we refer to the review \cite{OlGo22}. In this paper, we prove the existence and uniqueness of a solution to the coupled stochastic-differential reaction--diffusion model and present some numerical experiments in two space dimensions.

The evolution of the positions $X_1,\ldots,X_m$ of the neurite tips, starting from the stationary soma, is described by the stochastic differential equations
\begin{align}\label{1.X}
  \dd X_j(t) = \sum_{k=1}^n b_{jk}\big(c(X_j,t),t\big)
  \na c_k(X_j,t)\dd t + \sigma_j\dd W_j(t), \quad t>0,\ j=1,\ldots,m,
\end{align}
where the chemoattractant concentrations $c_1,\ldots,c_n$ are modeled by the reaction--diffusion equations
\begin{align}\label{1.c}
  \pa_t c_i - D_i\Delta c_i + \lambda_i c_i 
  = \sum_{j=1}^m a_{ij}\big(c(X_j,t),t\big)\eta(x-X_j(t))
  \quad\mbox{in }\R^d,\ t>0,\ i=1,\ldots,n,
\end{align}
together with the initial conditions
\begin{align}\label{1.ic}
  X_j(0) = X_j^0, \quad c_i(0) = c_i^0\quad\mbox{in }\R^d,\ 
  j=1,\ldots,m,\ i=1,\ldots,n,
\end{align}
where $(W_j(t))_{t\ge 0}$ are $d$-dimensional Brownian motions and $\eta$ is a mollifier approximating the Delta distribution. The parameters are the emission rates $a_{ij}$, the weights $b_{jk}$, the stochastic diffusion coefficients $\sigma_j>0$, the chemical diffusion coefficients $D_i>0$, and the chemical degradation rates $\lambda_i>0$. The function $c$ is the solution vector $(c_1,\ldots,c_n)$. 

The source term in \eqref{1.c} is modeled in \cite{SeBe00} by the Delta distribution $\delta_0$. There are several reasons why we choose an approximation of $\delta_0$. First, it seems more realistic to assume that the random walkers emit the chemical signals close to their position, i.e.\ in a nonlocal way. Second, it is shown in \cite{FOY21} that equation \eqref{1.c} with singular source term is singular at the location where we want to evaluate the gradient. More specifically, when we replace the right-hand side of \eqref{1.c} by a combination of $\delta_0(x-X_j(t))$, we face the following regularity issue. Since $\delta_0\in H^{-s}(\R^d)$ for $s>d/2$, parabolic regularity yields $\na c_i(t)\in H^{1-s}(\R^d)$, which is not sufficient to solve \eqref{1.X}, as this equation requires a H\"older continuous drift \cite{FGP10}. It is proven in \cite[Theorem 5]{TaYa13} that if the path of the point source is H\"older continuous with index $\alpha>1/2$, the solution to the heat equation with Dirac delta source term is singular. In our situation, the path is slightly rougher since the Brownian motion is H\"older continuous with $\alpha<1/2$ only. Third, we show in Section \ref{sec.eps} in a simplified deterministic setting that the positions become stationary when the mollifier approaches $\delta_0$. These comments justify the use of the mollifier in \eqref{1.c}. 

There exist many models in the literature describing neurite growth. For instance, a stochastic model for the position of the tip of an axon (a specific type of neurite) was suggested in \cite{PCLD11}. The angle of the vector relative to the axon orientation solves a stochastic differential equation. In the work \cite{Gaf06}, the motion of the growth cone is governed by an ordinary differential equation with a source term proportional to the gradient of the chemical signal, which solves a diffusion equation. The bundling of axons and their motion towards the targets are modeled by diffusion and differential equations in \cite{HeVo99}. A time-discrete model for the position and angle of the tip, coupled with a discrete-time neural network model, describing the chemotaxis behavior of the nematode C.\ elegans, was considered in \cite{XuDe10}. In the work \cite{YFBS21}, the dynamics of the growth cones on surfaces with micropatterned periodic features are described by Fokker--Planck equations. The analysis of the system \eqref{1.X}--\eqref{1.ic} seems to be new in the literature. For a review on growth cone chemotaxis and wiring of the nervous system, we refer to \cite{MFPRG06}. 

Our main analytical result is the existence of a unique solution (strong in the stochastic sense and classical in the PDE sense) to \eqref{1.X}--\eqref{1.ic}; see Section \ref{sec.ex}. The equations are discretized by an Euler--Maruyama Galerkin approximation in Section \ref{sec.num}. In Section \ref{sec.eps}, we discuss the behavior of the solution to the parabolic equation depending on the magnitude of the regularization of the source term in a simplified deterministic setting.


\section{Existence and uniqueness of solutions}\label{sec.ex}

The aim of this section is to prove that there exists a unique solution to \eqref{1.X}--\eqref{1.ic}. Our main result reads as follows.

\begin{theorem}[Existence and uniqueness]\label{thm.ex}
Let the following assumptions hold:
\begin{itemize}
\item[\rm (A1)] Let $\sigma_j$, $D_i$, $\lambda_i$ be positive for $i=1,\ldots,n$, $j=1,\ldots,m$ and let $T>0$. 
\item[\rm (A2)] Let $X_1^0,\ldots,X_m^0$ be independent and identically distributed square-integrable random variables on the probability space $(\Omega,\mathcal{F},\mathcal{P})$ and $c_1^0,\ldots,c_n^0\in L^2(\R^d)$.
\item[\rm (A3)] Let $a_{ij}:\R^n\times[0,T]\to\R$ and $b_{jk}:\R^n\times[0,T]\to\R$ be Lipschitz continuous and bounded, where $i,k=1,\ldots,n$, $j=1,\ldots,m$.
\item[\rm (A4)] Let $\eta\in C^{0,1}(\R^d)\cap L^1(\R^d)$. 
\end{itemize}
Then there exists a unique solution (strong in the stochastic sense and classical in the PDE sense) to system \eqref{1.X}--\eqref{1.ic}. 
\end{theorem}

The idea of the proof is to consider first the diffusion equation 
\begin{align}\label{2.u}
  \pa_t u - D\Delta u + \lambda u = f(t)\eta(x-\xi(t)), \quad t>0,
  \quad u(0)=u^0\quad\mbox{in }\R^d,
\end{align}
where $f\in C^0([0,\infty))$, $\xi$ is a continuous path on $\R^d$, $u^0\in L^2(\R^d)$, and $D>0$, $\lambda>0$. We show that $\na u$ is Lipschitz continuous with respect to $\xi$. This defines the Lipschitz continuous mapping $L:\xi\mapsto\na u$. Then the stochastic equation $\dd X = b(t)\na u(X,t)\dd t + \sigma\dd W(t)$ can be interpreted as
\begin{align*}
  \dd X = b(t)L(X)\dd t + \sigma\dd W(t), \quad t>0, \quad
  X(0) = X^0.
\end{align*}
Since $L$ is Lipschitz continuous, we conclude the existence of a solution $X$ by standard theory \cite{Pro05} and the result follows.

We proceed with the detailed proof of Theorem \ref{thm.ex}. For this, we start with some preparations and introduce the Gaussian heat kernel associated to \eqref{2.u}:
\begin{align}\label{2.Phi}
  \Phi_{D,\lambda}(x,t) = \frac{1}{(4\pi Dt)^{d/2}}
  \exp\bigg(-\frac{|x|^2}{4Dt} - \lambda t\bigg), \quad
  (x,t)\in\R^d\times(0,\infty).
\end{align}
We need the following property of the heat kernel.

\begin{lemma}\label{lem.gradPhi}
There exists $C>0$ only depending on the space dimension $d$ such that, for any $t>0$,
\begin{align*}
  \int_0^t\int_{\R^d}|\na\Phi_{D,\lambda}(z,t-s)|\dd z\dd s \le Ct^{1/2}.
\end{align*}
\end{lemma}

\begin{proof}
We transform $\tau=t-s$ and then $y=z/\sqrt{2D\tau}$ (with $\dd z=(2D\tau)^{d/2}\dd y$), giving
\begin{align*}
  I &:= \int_0^t\int_{\R^d}|\na\Phi_{D,\lambda}(z,t-s)|\dd z\dd s 
  = \frac{1}{(4\pi D)^{d/2}}\int_0^t\int_{\R^d}
  \frac{e^{-|z|^2/(4D(t-s))-\lambda(t-s)}}{(t-s)^{d/2+1}}
  \frac{|z|}{2D}\dd z\dd s \\
  &= \frac{1}{2D(4\pi D)^{d/2}}\int_0^t\int_{\R^d}
  \tau^{-d/2-1}e^{-|z|^2/(4D\tau)-\lambda\tau}|z|\dd z\dd\tau \\
  &= \frac{1}{(2\pi)^{d/2}}\int_0^t \tau^{-1/2}e^{-\lambda\tau}\dd\tau
  \int_{\R^d}e^{-|y|^2/2}|y|\dd y \le Ct^{1/2},
\end{align*}
where $C>0$ is a constant only depending on the dimension $d$.
\end{proof}

In the following, $C>0$ denotes a constant whose value may change from line to line. According to the Duhamel principle, the solution to \eqref{2.u} is uniquely given by 
\begin{align*}
  u(x,t) &= \int_0^t\int_{\R^d}f(s)\eta(y-\xi(s))
  \Phi_{D,\lambda}(x-y,t-s)\dd y\dd s 
  + \int_{\R^d}u^0(x-y)\Phi_{D,\lambda}(y,t)\dd y \\
  &= \int_0^t\int_{\R^d}f(s)\eta(x-y)
  \Phi_{D,\lambda}(y-\xi(s),t-s)\dd y\dd s 
  + \int_{\R^d}u^0(x-y)\Phi_{D,\lambda}(y,t)\dd y,
\end{align*}
where the second formulation follows from exchanging the order of the convolution. This result can be proved in a similar way as in \cite[Sec.~2.3.1c, Theorem 2]{Eva98}. It yields a formula for the solution to the linear equation associated to \eqref{1.c}. To simplify the notation, we set $\Phi_i:=\Phi_{D_i,\lambda_i}$ for $i=1,\ldots,n$. These comments yield the existence of a solution to the following linear problem.

\begin{lemma}
Let $f_{ij}\in C^0([0,\infty))$, $\xi_j\in C^0(\R^d)$ for $i=1,\ldots,n$, $j=1,\ldots,m$, $\eta\in C^{0,1}(\R^d)\cap L^1(\R^d)$, and $v_i^0\in L^1(\R^d)$. Then the functions
\begin{align*}
  v_i(x,t) = \int_0^t\int_{\R^d}\sum_{j=1}^m f_{ij}(s)\eta(y-\xi_j(s))
  \Phi_i(x-y,t-s)\dd y\dd s
  + \int_{\R^d} v_i^0(x-y)\Phi_i(y,t)\dd y
\end{align*}
are continuous in $\R^d\times(0,\infty)$ and uniquely solve
\begin{align*}
  \pa_t v_i - D_i\Delta v_i + \lambda_i v_i
  = \sum_{j=1}^m f_{ij}(t)\eta(x-\xi_j(t))\quad\mbox{in }\R^d,\ t>0,
\end{align*}
with initial condition $v_i(0)=v_i^0$ in $\R^d$, $i=1,\ldots,n$.
\end{lemma}

For given continuous paths $\xi_1,\ldots,\xi_m$, we wish to solve the nonlinear equations \eqref{1.c}. 

\begin{lemma}
Let $\xi_1,\ldots,\xi_m$ be continuous paths on $\R^d$, let $T>0$, and let $a_{ij}:\R^n\times[0,T]\to\R$ be Lipschitz continuous and bounded. Then the system
\begin{align}\label{2.ui}
  \pa_t u_i - D_i\Delta u_i + \lambda_i u_i
  = \sum_{j=1}^m a_{ij}\big(u(\xi_j(t),t),t\big)\eta(x-\xi_j(t))
  \quad\mbox{in }\R^d,\ t>0, 
\end{align}
has a unique continuous solution $u_i(x,t;\xi):=u_i(x,t)$ satisfying $u_i(0)=u_i^0$ in $\R^d$ for $i=1,\ldots,n$, recalling that $\Phi_i=\Phi_{D_i,\lambda_i}$, where $\Phi_{D_i,\lambda_i}$ is given by \eqref{2.Phi}.
\end{lemma}

\begin{proof}
The result follows from Banach's fixed-point theorem. For this, we introduce for some $T>0$ the space
\begin{align*}
  Z_T = C^0(\R^d\times(0,T);\R^n)\cap L^\infty(\R^d\times(0,T);\R^n)
  \cap L^1(\R^d\times(0,T);\R^n),
\end{align*}
endowed with the $L^\infty(\R^d\times(0,T);\R^n)$ norm, and define the functional $F_i:Z_T\to Z_T$ by
\begin{align*}
  F(u)(x,t) &= \int_0^t\int_{\R^d}\sum_{j=1}^m 
  a_{ij}\big(u(\xi_j(s),s),s\big)
  \eta(y-\xi_j(s))\Phi_i(x-y,t-s)\dd y\dd s \\
  &\phantom{xx}+ \int_{\R^d}u^0_i(x-y)\Phi_i(y,t)\dd y
\end{align*}
for $u=(u_1,\ldots,u_n)\in Z_T$ and $(x,t)\in\R^d\times(0,T)$. To show that $F_i$ is a contraction on $Z_T$ for some sufficiently small $T>0$, we use the Lipschitz continuity of $a_{ij}$ and compute for $u$, $\bar{u}\in Z_T^n$:
\begin{align*}
  \|F_i(u)-F_i(\bar{u})\|_{Z_T}
  &\le \sup_{(x,t)\in\R^d\times(0,T)}\int_0^t\int_{\R^d}\sum_{j=1}^m
  \big|a_{ij}\big(u(\xi_j(s),s),s\big) 
  - a_{ij}\big(\bar{u}(\xi_j(s),s),s\big)\big| \\
  &\phantom{xx}\times|\eta(y-\xi_j(s))|\Phi_i(x-y,t-s)\dd y\dd s \\
  &\le C\sup_{(x,t)\in\R^d\times(0,T)}\int_0^t\sum_{j=1}^m
  |u(\xi_j(s),s)-\bar{u}(\xi_j(s),s)| \\
  &\phantom{xx}\times\int_{\R^d}
  |\eta(y-\xi_j(s))|\Phi_i(x-y,t-s)\dd y\dd s \\
  &\le C\int_0^t\|u-\bar{u}\|_{Z_T}\dd s
  \le CT\|u-\bar{u}\|_{Z_T}.
\end{align*}
Thus, choosing $T>0$ such that $CT<1$, $F_i$ is a contraction on $Z_T$, and Banach's fixed-point theorem gives the existence of a unique continuous solution to \eqref{2.ui} on $(0,T)$. We can extend the solution for all $t\in(0,T)$ by repeating the argument. 
\end{proof}

The solution $u_i(x,t;\xi)$ is Lipschitz continuous with respect to $(x,\xi)$, where $\xi=(\xi_1,\ldots,$ $\xi_m)$.

\begin{lemma}\label{lem.Lipu}
There exists a constant $C(t)$, which is increasing and continuous in $t$, such that for all continuous paths $\xi$, $\bar\xi$ on $\R^d$ and points $x$, $\bar{x}\in\R^d$,
\begin{align*}
  |u_i(x,t;\xi)-u_i(\bar{x},t;\bar\xi)|
  \le C(t)\big(|x-\bar{x}| + \|\xi-\bar\xi\|_{C^0([0,t])}\big).
\end{align*}
\end{lemma}

\begin{proof}
By the triangle inequality,
\begin{align}\label{2.aux}
  & |u_i(x,t;\xi)-u_i(\bar{x},t;\bar\xi)| \le I_1+I_2, \quad
  \mbox{where} \\
  & I_1 = |u_i(x,t;\xi)-u_i(\bar{x},t;\xi)|, \quad
  I_2 = |u_i(\bar{x},t;\xi)-u_i(\bar{x},t;\bar\xi)|.
  \nonumber 
\end{align}
We use the Lipschitz continuity of $\eta$ and the boundedness of $a_{ij}$ to find that
\begin{align*}
  I_1 &\le \int_0^t\int_{\R^d}\sum_{j=1}^m
  |\eta(x-y)-\eta(\bar{x}-y)|
  |a_{ij}(u(\xi_j(x),s;\xi)|\Phi_i(y-\xi_j(s),t-s)\dd y\dd s \\
  &\le C|x-\bar{x}|\sum_{j=1}^m\int_0^t\int_{\R^d}
  \Phi_i(y-\xi_j(s),t-s)\dd y\dd s \\
  &\le C|x-\bar{x}|\sum_{j=1}^m\int_0^t\int_{\R^d}
  \Phi_i(z,t-s)\dd z\dd s \le Ct|x-\bar{x}|.
\end{align*}
For $I_2$, we use the Lipschitz continuity of $\eta$ and $a_{ij}$:
\begin{align*}
  I_2 &\le \int_0^t\int_{\R^d}\sum_{j=1}^m
  \big|a_{ij}\big(u(\xi_j(s),s;\xi),s\big)\eta(y-\xi_j(s))
  - a_{ij}\big(u(\bar{\xi}_j(s),s;\bar\xi),s\big)
  \eta(y-\bar{\xi}_j(s))\big| \\
  &\phantom{xx}\times\Phi_i(\bar{x}-y,t-s)\dd y\dd s \\
  &\le \int_0^t\int_{\R^d}\sum_{j=1}^m
  \big|\eta(y-\xi_j(s))-\eta(y-\bar{\xi}_j(s))\big|
  \big|a_{ij}\big(u(\xi_j(s),s;\xi),s\big)\big|
  \Phi_i(\bar{x}-y,t-s)\dd y\dd s \\
  &\phantom{xx}+ \int_0^t\int_{\R^d}\sum_{j=1}^m
  \big|a_{ij}\big(u(\xi_j(s),s;\xi),s\big)
  - a_{ij}\big(u(\bar{\xi}_j(s),s;\bar\xi),s\big)\big| \\
  &\phantom{xx+}
  \times|\eta(y-\bar{\xi}_j(s))|\Phi_i(\bar{x}-y,t-s)\dd y\dd s \\
  &\le Ct\|\xi-\bar\xi\|_{C^0(0,t)} + C\int_0^t\sum_{j=1}^m
  \big|u(\xi_j(s),s;\xi) - u(\bar{\xi}_j(s),s;\bar\xi)\big|\dd s.
\end{align*}
Summing \eqref{2.aux} over $i=1,\ldots,n$, we infer that
\begin{align}\label{2.aux2}
  \big|u(x,t;\xi)-u(\bar{x},t;\bar\xi)\big|
  &\le Ct\big(|x-\bar{x}| + \|\xi-\bar\xi\|_{C^0([0,t])}\big) \\
  &\phantom{xx}+ C\int_0^t\sum_{j=1}^m
  \big|u(\xi_j(s),s;\xi) - u(\bar{\xi}_j(s),s;\bar\xi)\big|\dd s.
  \nonumber
\end{align}
We replace in this inequality $(x,\bar{x})$ by $(\xi_j(s),\bar{\xi}_j(s))$ and apply Gronwall's lemma:
\begin{align*}
  \sum_{j=1}^m\big|u(\xi_j(t),t;\xi) 
  - u(\bar{\xi}_j(t),t;\bar\xi)\big| \le Ct(1+e^{Ct})
  \|\xi-\bar\xi\|_{C^0([0,t])}.
\end{align*}
Finally, we use this estimate in \eqref{2.aux2}:
\begin{align*}
  |u(&x,t;\xi)-u(\bar{x},t;\bar\xi)|
  \le Ct\big(|x-\bar{x}| + \|\xi-\bar\xi\|_{C^0([0,t])}\big) \\
  &\phantom{xx}+ C\|\xi-\bar\xi\|_{C^0([0,t])}\int_0^t s(1+e ^{Cs})\dd s
  \le C(t)\big(|x-\bar{x}| + \|\xi-\bar\xi\|_{C^0([0,t])}\big),
\end{align*}
finishing the proof.
\end{proof}

Next, we show that $\na u_i(x,t;\xi)$ is also Lipschitz continuous with respect to $(x,\xi)$.

\begin{lemma}\label{lem.LipDu}
It holds for all continuous paths $\xi$, $\bar\xi$ on $\R^d$ and points $x$, $\bar{x}\in\R^d$ that
\begin{align*}
  |\na u_i(x,t;\xi)-\na u_i(\bar{x},t;\bar\xi)|
  \le C(t)t^{1/2}\big(|x-\bar{x}| + \|\xi-\bar\xi\|_{C^0([0,t])}\big),
\end{align*}
where $C(t)>0$ depends on time. 
\end{lemma}

\begin{proof}
We estimate similarly as in the proof of Lemma \ref{lem.Lipu}:
\begin{align}\label{2.aux3}
  & |\na u_i(x,t;\xi)-\na u_i(\bar{x},t;\bar\xi)|
  \le I_3 + I_4, \quad\mbox{where} \\
  & I_3 = |\na u_i(x,t;\xi)-\na u_i(\bar{x},t;\xi)|, \quad
  I_4 = |\na u_i(\bar{x},t;\xi)-\na u_i(\bar{x},t;\bar\xi)|.
  \nonumber 
\end{align}
We need the boundedness of $a_{ij}$, the Lipschitz continuity of $\eta$, and Lemma \ref{lem.gradPhi} to estimate
\begin{align*}
  I_3 &\le \int_0^t\int_{\R^d}\sum_{j=1}^m
  \big|a_{ij}\big(u(\xi_j(s),s;\xi),s\big)\big|
  |\eta(x-y)-\eta(\bar{x}-y)|
  |\na\Phi_i(y-\xi_j(s),t-s)|\dd y\dd s \\
  &\le C|x-\bar{x}|\int_0^t\int_{\R^d}\sum_{j=1}^m
  |\na\Phi_i(y-\xi_j(s),t-s)|\dd y\dd s \\
  &\le C|x-\bar{x}|\int_0^t\int_{\R^d}
  |\na\Phi_i(z,t-s)|\dd y\dd s \le Ct^{1/2}|x-\bar{x}|.
\end{align*}
Furthermore, using the estimate of Lemma \ref{lem.Lipu},
\begin{align*}
  I_4 &\le \int_0^t\int_{\R^d}\sum_{j=1}^m
  \big|a_{ij}\big(u(\xi_j(s),s;\xi),s\big)\eta(y-\xi_j(s))
  -a_{ij}\big(u(\bar{\xi}_j(s),s;\bar\xi),s\big)\eta(y-\bar{\xi}_j(s))
  \big| \\
  &\phantom{xx}\times
  |\na\phi_i(\bar{x}-y,t-s)|\dd y\dd s \\
  &\le \int_0^t\int_{\R^d}\sum_{j=1}^m
  \big|a_{ij}\big(u(\xi_j(s),s;\xi),s\big)
  -a_{ij}\big(u(\bar{\xi}_j(s),s;\bar\xi),s\big)\big| \\
  &\phantom{xxxx}\times |\eta(y-\xi_j(s))|
  |\na\phi_i(\bar{x}-y,t-s)|\dd y\dd s \\
  &\phantom{xx}+ \int_0^t\int_{\R^d}\sum_{j=1}^m
  \big|a_{ij}\big(u(\bar{\xi}_j(s),s;\bar{\xi}),s\big)\big|
  \big|\eta(y-\xi_j(s))-\eta(y-\bar{\xi}_j(s))\big| \\
  &\phantom{xxxx}\times|\na\phi_i(\bar{x}-y,t-s)|\dd y\dd s \\
  &\le \int_0^t\int_{\R^d}\sum_{j=1}^m
  \Big(\big|u(\xi_j(s),s;\xi)-u(\bar{\xi}_j(s),s;\bar\xi)\big|
  + |\xi_j(s)-\bar{\xi}_j(s)|\Big) \\
  &\phantom{xx}\times|\na\phi_i(\bar{x}-y,t-s)|\dd y\dd s
  \le C(t)t^{1/2}\|\xi-\bar{\xi}\|_{C^0([0,t])}.
\end{align*}
Combining the estimates for $I_3$ and $I_4$ in \eqref{2.aux3} concludes the proof.
\end{proof}

We are now in the position to prove the existence result. 

\begin{proof}[Proof of Theorem \ref{thm.ex}]
Let $X=(X_1,\ldots,X_m)$ be a continuous path in $\R^{dm}$ and let $c=c(x,t;X)$ be the unique solution to \eqref{1.c} with initial data in \eqref{1.ic}. Lemma \ref{lem.LipDu} shows that for some $t>0$ the mapping 
\begin{align*}
  \R^d\times C^0([0,t];\R^{dm})\to \R^{dn}, \quad 
  (x,X)\mapsto \na c(x,t;X)
\end{align*}
is Lipschitz continuous. Moreover, by Lemma \ref{lem.Lipu}, the mapping $(x,X)\mapsto c(x,t;X)$ is Lip\-schitz continuous. Consequently, 
\begin{align*}
  (x,X)\mapsto \sum_{k=1}^m b_{jk}\big(c(X_j(t),t;X),t\big)
  \na c(X_j(t),t;X)
\end{align*}
is Lipschitz continuous, since $b_{jk}$ is bounded by assumption and $\na c$ is bounded as a consequence of Lemma \ref{lem.LipDu}. Thus, system \eqref{1.X}--\eqref{1.c} can be reduced to a system of stochastic differential equations with a drift term that is Lipschitz continuous with respect to $X$. We infer from \cite[Chap.~5, Theorem 7]{Pro05} the existence of a unique strong solution $X\in C^0([0,\infty);\R^{dm})$ to \eqref{1.X}--\eqref{1.ic}. This proves Theorem \ref{thm.ex}. 
\end{proof}


\section{Numerical scheme}\label{sec.num}

We discretize the reaction--diffusion equations \eqref{1.c} in space by a Galerkin method. To this end, let $\Omega=(-L,L)^2\subset\R^2$ for some $L>0$ with periodic boundary conditions and let $V_h\subset H^1(\Omega)$ be a finite-dimensional subspace with basis $(v_1,\ldots,v_N)$. The Galerkin approximation of $c_i$ is given by $c_i = \sum_{k=1}^N q_{ik} v_k$. Using this approximation in \eqref{1.c} with the test function $v_\ell$ yields the differential system
\begin{align}\label{3.q}
  Mq_i' + B_iq_i = f_i(X(t)), \quad t>0,\ i=1,\ldots,n,
\end{align}
where $q_i=(q_{i1},\ldots,q_{iN})$, $M=(M_{jk})$, $B_i=(B_{ijk})$, $f_i=(f_{i1},\ldots,f_{iN})$, and
\begin{align*}
  & M_{jk} = \int_\Omega v_jv_k\dd x, \quad
  B_{ijk} = \int_\Omega(D_i\na v_j\cdot\na v_k + \lambda v_jv_k)\dd x, \\
  & f_{ik}(X(t)) = \sum_{j=1}^m\int_\Omega a_{ij}\big(c(X_j(t),t),t\big)
  \eta(x-X_j(t))v_k(x)\dd x
\end{align*}
for $i=1,\ldots,n$ and $k=1,\ldots,N$. The stochastic differential equation \eqref{1.X} becomes
\begin{align}\label{3.X}
  \dd X_j(t) = \sum_{k=1}^n\sum_{\ell=1}^N b_{jk}(c(X_j),t)
  q_{k\ell}(t)\na v_\ell(X_j,t)\dd t + \sigma_j\dd W_j(t), \quad
  t>0,\ j=1,\ldots,m.
\end{align}
We choose the mollifier $\eta(x)=\eta_\eps(x)=(4\pi\eps)^{-1}\exp(-|x|^2/(4\eps))$ for $x\in\R^2$. We discretize in time by solving \eqref{3.X} by the Euler--Maruyama scheme and \eqref{3.q} by the implicit Euler method on the partition $t_\alpha=\alpha\Delta t$ for $\alpha\ge 0$ and some $\Delta t>0$:
\begin{align}\label{3.Xa}
  X_j^{[\alpha+1]} &= X_j^{[\alpha]} + \sum_{k=1}^n\sum_{\ell=1}^N 
  b_{jk}(c(X_j^{[\alpha]}))q_{k\ell}(t_\alpha)
  \na v_\ell(X_j^{[\alpha]},t_\alpha)\Delta t 
  + \sigma_j\Delta W_j, \\
  Mq_i^{[\alpha+1]} &= Mq_i^{[\alpha]} - \Delta t B_iq_i^{[\alpha+1]}
  + \Delta t f_i(X^{[\alpha+1]}), \label{3.qa}
\end{align}
where $\Delta W_j$ is standard normally distributed.

We use linear finite elements with NGSolve on a uniform partition of $\Omega=(-L,L)^2$ with size $L=3$ and space grid $\Delta x=0.05$. The time step size equals $\Delta t = 0.001$. If a growth cone is close to a soma or another growth cone (we use the threshold 0.1), we stop its movement and the emission of chemicals, i.e., it becomes inactive. The point sources are implemented by copying the source into every ``adjacent'' domain (discarding parts that have crossed the periodic boundary more than once). 

For the simulations, we consider a system inspired by \cite{SeBe00}. We choose three chemical substances: an attractive one ($c_1$), a repulsive one ($c_2$), and a triggering one ($c_3$), as well as the positions of the growth cones $X_1^{(1)},\ldots,X_g^{(1)}$ of the neurites and the somas $X_{1}^{(2)},\ldots,X_{s}^{(2)}$, where $g,s\in\N$ and $g+s=m$. The somas are stationary and emit the repulsive cue. If they sense a sufficiently large density of the triggering agent, they also emit the attractive substance. The growth cones emit the triggering agent and, when they measure a sufficiently large density of the triggering cue, also the attractive substance. Moreover, they respond to repulsive cues for small times and to attractive cues at later times. The precise definitions are
\begin{align*}
	a_{1j}^{(1)}(c(X_{j}^{(1)}(t),t)) &= 15\arctan\big(2.25c_3(X_j^{(1)}(t)-3.5,t)\big)
  + \frac{15}{2}\pi + 15\arctan(2t), \\
	a_{1k}^{(2)}(c(X_{k}^{(2)}(t),t)) &= 5\arctan\big(0.5c_3(X_{k}^{(2)}(t),t)\big), \\
	a_{3j}^{(1)}(c(X_{j}^{(1)}(t),t)) &= 20\arctan\big(2.25c_1(X_{j}^{(1)}(t)-3,t)\big)
  + \frac{15}{2}\pi + 15\arctan(2t), \\
	a_{2j}^{(1)}(c(X_{j}^{(1)}(t),t)) &= a_{3k}^{(2)}(c(X_{k}^{(2)}(t),t)) = 0, \quad
	a_{2k}^{(2)}(c(X_{k}^{(2)}(t),t)) = 3, \\
  b_{j1}^{(1)}(c(X_{j}^{(1)}(t),t)) &= \frac{\beta}{\pi}\arctan\big(0.3(10t-13)\big)
  + \frac{\beta}{2}, \\
  b_{j2}^{(1)}(c(X_{j}^{(1)}(t),t)) &= \frac{\gamma}{\pi}\arctan\big(0.3(10t-13)\big)
  - \frac{\gamma}{2}, \\
  b_{j3}^{(1)}(c(X_{j}^{(1)}(t),t)) &= 0, \\
  b_{k1}^{(2)}(c(X_{k}^{(2)}(t),t))&= b_{k2}^{(2)}(c(X_{k}^{(2)}(t),t)) = b_{k2}^{(2)}(c(X_{k}^{(2)}(t),t)) = 0 \\
  \text{for} \quad j=1,\ldots,g&, \quad k = 1, \ldots, s.
\end{align*} 
The function $a_{1j}^{(1)}$ produces more of the attractive cue in the presence of the triggering agent and increases with time; see Figure \ref{fig.ab}. Similarly, $a_{3j}^{(1)}$ describes an increased production of the triggering agent in response to the attractive cue. The factors $\beta$ and $\gamma$ determine the strength of the attractive and repulsive cue, respectively. For small times, the values of $b_{j1}^{(1)}$ are very small. This means that initially the walkers move relatively freely and respond to the presence of other walker only after some time has passed. During the initial phase, the walkers are mostly guided by the repulsive cue, which is constantly produced by the soma ($a_{2j}^{(2)}=3$). Thus, the walkers move away from the soma before responding to the attractive cue.  

\begin{figure}[ht]
\includegraphics[width=65mm]{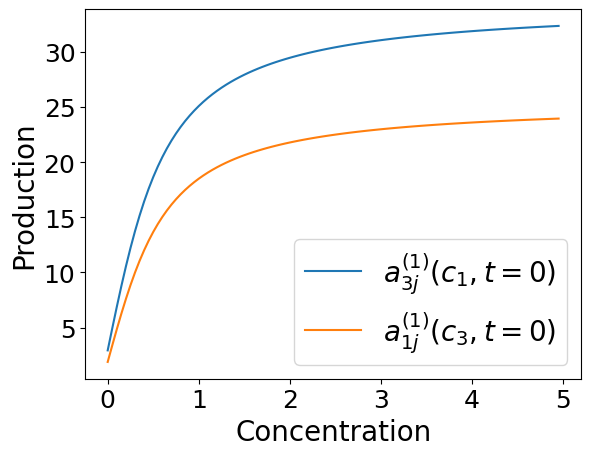}
\includegraphics[width=65mm]{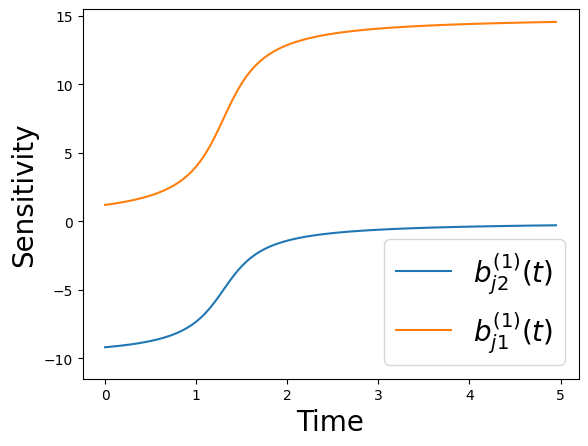}
\caption{Coefficient functions $a_{1j}^{(1)}$, $a_{3j}^{(1)}$ (left) and $b_{1j}^{(1)}$, $b_{2j}^{(1)}$ (right).}
\label{fig.ab}
\end{figure}

\subsection*{First experiment}
We place $s=9$ somas on a $3\times 3$ grid with randomly distributed spatial deviation from their grid positions (deviation $=0.3$) in the domain (red dots in Figure \ref{fig.9}). We have simulated $g=9$ neurites (left panel) and $g = 27$ neurites (right panel). In the latter case, three neurites originate from each soma. To prevent neurites originating from the same soma from connecting with each other, their initialization times are offset by $t=0.8$. In this first experiment, due to the large number of neurites, we have reduced the production of attractive cues and increased the production of repulsive cues according to 
\begin{align*}
  a_{1j}^{(1)}(c(X_{j}^{(1)}(t),t)) 
  &= 5\arctan\big(2.25c_3(X_j^{(1)}(t)-3.5,t)\big)
  + \frac{5}{2}\pi + 5\arctan(2t), \\
  a_{2k}^{(2)}(c(X_{k}^{(2)}(t),t)) &= 5, 
  \quad j = 1,\ldots,g,\ k = 1, \ldots ,s.
\end{align*}
The growth cones move according to the discrete equations \eqref{3.Xa}--\eqref{3.qa}, yielding the black curves in Figure \ref{fig.9}. We have chosen $\sigma:=\sigma_j=0.2$ for $j=1,\ldots,m$, $\beta=15$, $\gamma=10$, $\eps=0.01$, and $T=5$. Whenever two growth cones approach, their movement is stopped. The chosen parameters enable the additional neurites to extend in new directions and form extra connections.

\begin{figure}[ht]
\includegraphics[width=70mm]{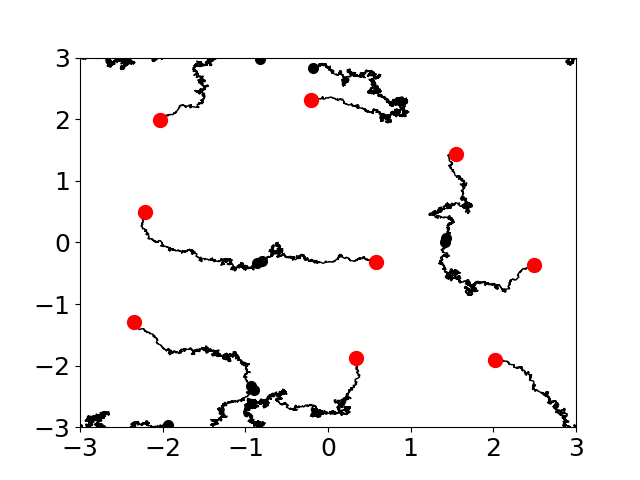}
\includegraphics[width=70mm]{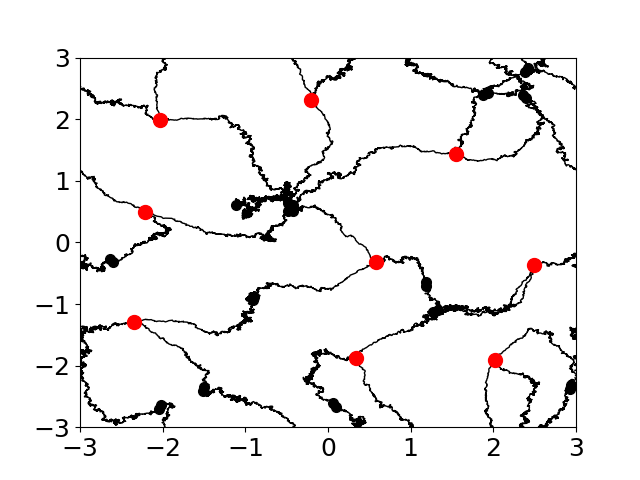}
\caption{Simulation of 9 somas and 9 neurites (left) and 27 neurites (right). The somas are distributed on a grid with random variation in their positions.}
\label{fig.9}
\end{figure}

\subsection*{Second experiment: variation of $\sigma$}

Next, we simulate six somas and six neurites for two different values of $\sigma$; see Figure \ref{fig.sigma}. The parameters are $\eps=0.1$, $\beta=15$, $\gamma=10$, and $T=5$. The somas are placed randomly and both simulations use the same realizations of the Brownian motions. When $\sigma$ is small, the motion of the growth cones is mainly governed by the drift, while for large values of the diffusion coefficient $\sigma$, the influence of the stochastic term increases and the paths of the neurite show random fluctuations. 

\begin{figure}[ht]
\includegraphics[width=70mm]{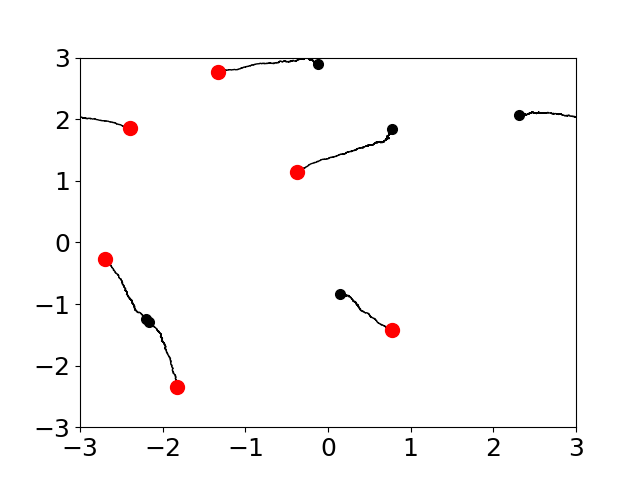}
\includegraphics[width=70mm]{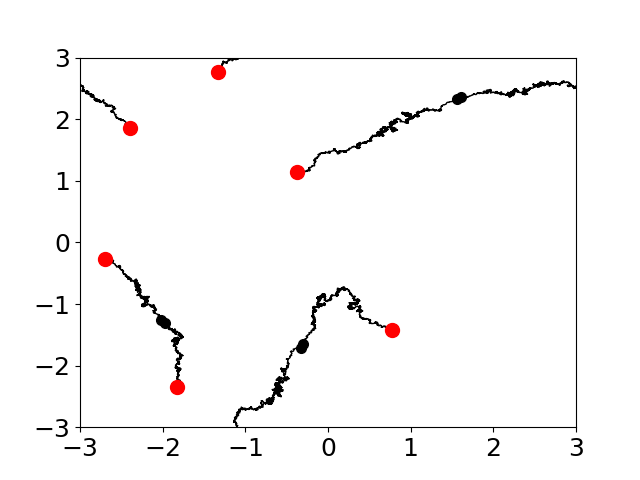}
\caption{Simulation of six neurites and six somas with diffusion coefficient $\sigma=0.05$ (left) and $\sigma=0.2$ (right).}
\label{fig.sigma}
\end{figure}

\subsection*{Third simulation: variation of $\eps$}

We consider various values of the parameter $\eps$, used in the Gaussian mollifier and shown in Figure \ref{fig.eps}. The production of the signaling agents becomes very localized for small values of $\eps$. As a result the signals spread in a small neighborhood of the tip position, and the neurite movement stalls. For larger values of $\eps$, the neurites are able to connect. We discuss the role of $\eps$ in Section \ref{sec.eps} in more detail. 

\begin{figure}[ht]
\includegraphics[width=53mm,height=50mm]{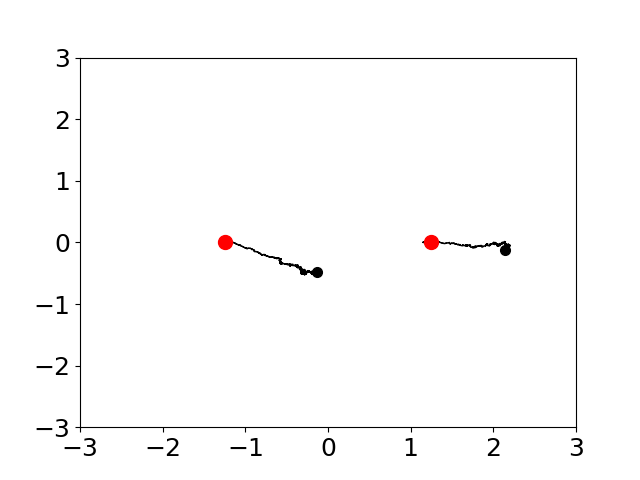}
\includegraphics[width=53mm,height=50mm]{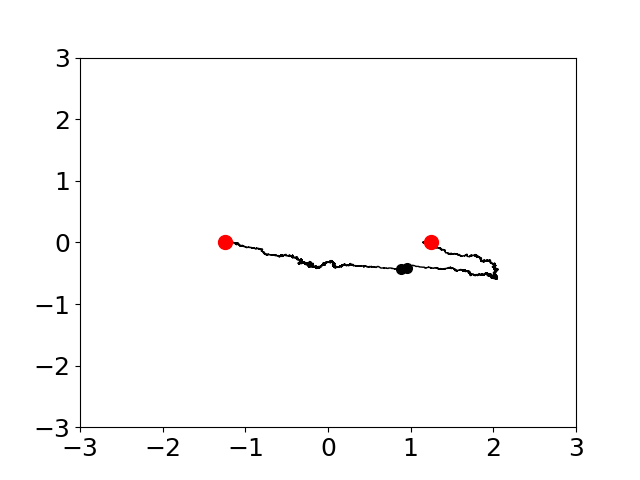}
\includegraphics[width=53mm,height=50mm]{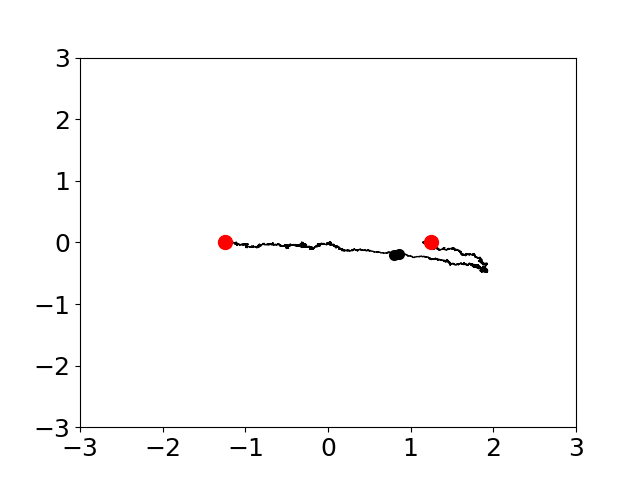}
\caption{ Simulation of two neurites and two somas using $\sigma=0.1$ and $\eps=0.005$ (left), $\eps=0.02$ (middle), and $\eps=0.04$ (right).}
\label{fig.eps}
\end{figure}

\subsection*{Fourth simulation: variation of $\beta$ and $\gamma$}

We study the influence of the attractive strength $\beta$ and repulsive strength $\gamma$ in Figure \ref{fig.beta}. Larger attraction has the effect that the neurites stay closer together and leads to a reduced mobility, while growth cones travel larger distances when the repulsion force increases. 

\begin{figure}[ht]
\includegraphics[width=65mm]{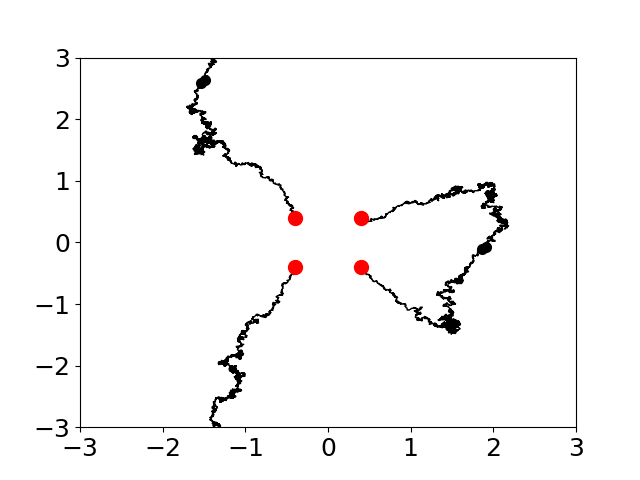}
\includegraphics[width=65mm]{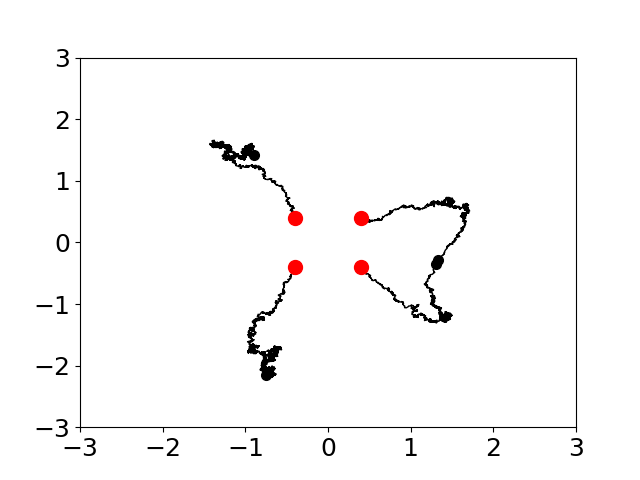}
\includegraphics[width=65mm]{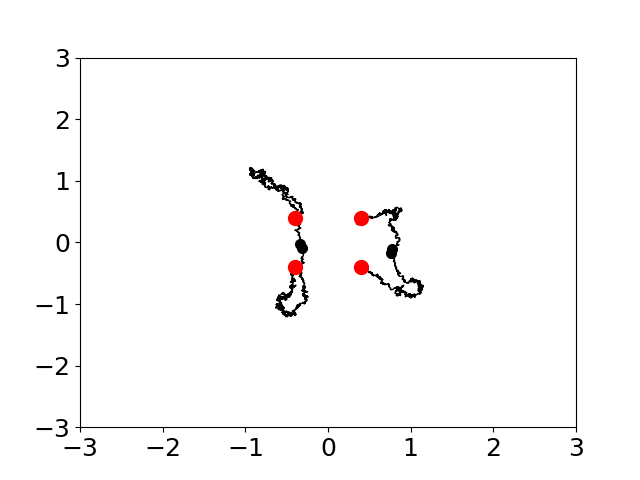}
\includegraphics[width=65mm]{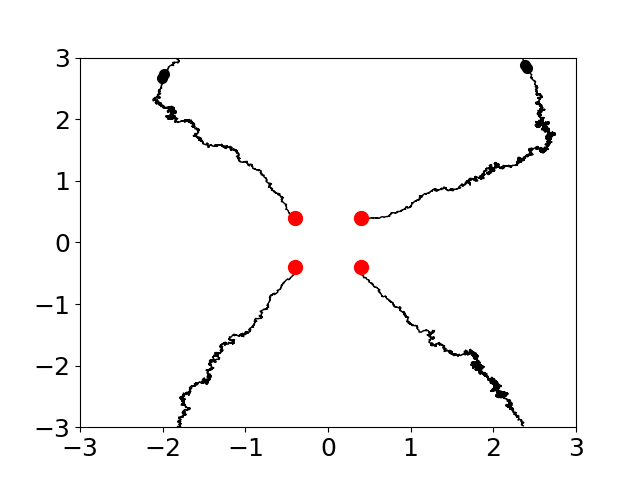}
\caption{Four neurites and four somas using $\sigma=0.2$, $\eps=0.01$, $T=5$, and $(\alpha,\beta)=(5,10)$ (top left), $(\alpha,\beta)=(20,10)$ (top right), $(\alpha,\beta)=(15,5)$ (bottom left), $(\alpha,\beta)=(15,20)$ (bottom right).}
\label{fig.beta}
\end{figure}


\section{Discussion of the limit $\eps\to 0$}\label{sec.eps}

In the previous section, we have used the mollifier
\begin{align*}
  \eta_\eps(x) = \frac{1}{(4\pi\eps)^{d/2}}
  \exp\bigg(-\frac{|x|^2}{4\eps}\bigg), \quad x\in\R^d.
\end{align*}
We observed that the movement of the growth cones slows down when $\eps$ becomes smaller. In this section, we discuss this behavior for a simplified deterministic system. For a given force in the differential equation for the position $X(t)$, we show that the velocity $\pa_t X$ converges to zero as $\eps\to 0$. This result means that a mollifier on the right-hand side of equation \eqref{1.c} for $c_i$ is necessary to obtain nontrivial dynamics and should not be replaced by the Delta distribution. 

Let $F=(u,0)\in\R^2$ be some constant force with $u>0$, let $a>0$, and consider the deterministic system
\begin{align}\label{4.eq}
  \pa_t X_\eps = \na c_\eps(X_\eps,t) + F, \quad 
  \pa_t c_\eps - \Delta c_\eps = a\eta_\eps(x-X_\eps(t))
  \quad\mbox{in }\R^d,\ t>0
\end{align}
with the initial data $X_\eps(0)=(X_\eps^{(1)}(0),X_\eps^{(2)}(0))=(x^0,y^0)$ with $y^0\in\R$ and $c_\eps(0)=c^0$. We claim that the dynamics of $X_\eps$ is purely one-dimensional. First, we compute the gradient $\na c_\eps$.

\begin{lemma}\label{lem.nac}
The gradient of the chemical signal is given by
\begin{align*}
  \na c_\eps(X(t),t) &= -\frac{a}{8\pi}\int_\eps^{t+\eps}
  \frac{1}{\tau^2}\big(X_\eps(t)-X_\eps(t+\eps-\tau)\big) \\
  &\times\exp\bigg(-\frac{|X_\eps(t)-X_\eps(t+\eps-\tau)|^2}{4\tau}
  \bigg)\dd \tau,
\end{align*}
\end{lemma}

Thus, $X_\eps^{(2)}(t)=y^0$ is a constant solution since $\pa c_\eps/\pa x_2=0$, and only the first component of $X_\eps$ changes over time.

\begin{proof}[Proof of Lemma \ref{lem.nac}]
With the heat kernel $\Phi_{1,0}$ (see \eqref{2.Phi}) and the mollifier $\eta_\eps(x)=\Phi_{1,0}(x,\eps)$, the solution $c_\eps$ to the second equation in \eqref{4.eq} reads as 
\begin{align*}
  c_\eps(x,t) &= a\int_0^t\int_{\R^d}\Phi_{1,0}(y-X_\eps(s),\eps)
  \Phi_{1,0}(x-y,t-s)\dd y\dd s \\
  &= a\int_0^t\int_{\R^d}\Phi_{1,0}(z,\eps)
  \Phi_{1,0}(x-X_\eps(s)-z,t-s)\dd y\dd s \\
  &= a\int_0^t\big(\Phi_{1,0}(\cdot,\eps)*\Phi_{1,0}(\cdot,t-s)\big)
  (x-X_\eps(s))\dd s \\
  &= a\int_0^t\Phi_{1,0}(x-X_\eps(s),t-s+\eps)\dd s,
\end{align*}
where the last step follows from the semigroup property (additivity) of the heat kernel. The gradient becomes
\begin{align*}
  \na c_\eps(x,t) = -a\int_0^t\frac{1}{4\pi(t-s+\eps)}
  \frac{x-X(s)}{2(t-s+\eps)}
  \exp\bigg(-\frac{|x-X_\eps(s)|^2}{4(t-s+\eps)}\bigg)\dd s.
\end{align*}
The result follows after transforming $\tau=t-s+\eps$.
\end{proof}

Numerical simulations show that the speed $\pa_t X_\eps^{(1)}$ of the walker decreases with increasing time and seems to approach some asymptotic speed; see Figure \ref{fig.vel}. The simulations are performed with $\Delta t=10^{-3}$, $\Delta x=10^{-2}$, and $a=1$. This behavior motivates the assumptions imposed in the following proposition.

\begin{figure}[ht]
\includegraphics[width=70mm]{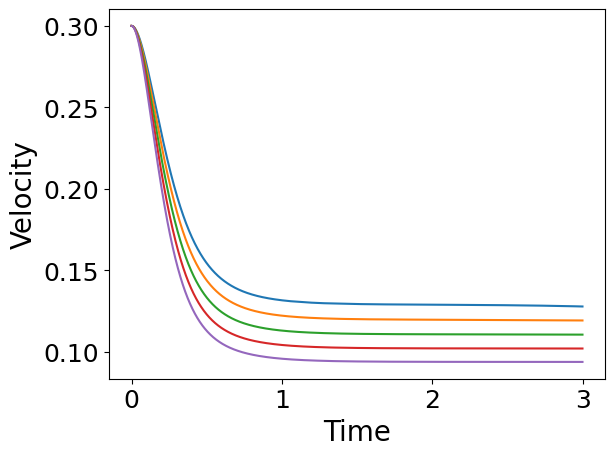}
\caption{Speed $\pa_t X_\eps(t)$ of a neurite for $\eps=0.05,0.047,\ldots,0.038$ (from top to bottom).}
\label{fig.vel}
\end{figure}

\begin{proposition}
Let $\eps>0$, $d=2$, and $(X_\eps,c_\eps)$ be the solution to \eqref{4.eq}. We assume that {\rm (i)} $t\mapsto\pa_t X_\eps^{(1)}(t)$ is nonincreasing for sufficiently large $t>0$, {\rm (ii)} there exists $v_\eps>0$ such that $\pa_t X_\eps^{(1)}(t)\to (v_\eps,0)$ as $t\to\infty$ and {\rm (iii)} $v_\eps$ is bounded for all $\eps\in(0,1)$. Then $v_\eps\to 0$ as $\eps\to 0$.
\end{proposition}

The proposition shows that equations \eqref{4.eq} becomes stationary in the limit $\eps\to 0$.

\begin{proof}
Assumption (ii) and the mean-value theorem show that for $\tau>\eps$,
\begin{align}\label{4.X}
  X_\eps^{(1)}(t) - X_\eps^{(1)}(t+\eps-\tau) 
  = \pa_t X_\eps^{(1)}(\xi_t)(\tau-\eps)
  \to v_\eps(\tau-\eps) \quad\mbox{as }t\to\infty,
\end{align}
where $\xi_t$ is a number between $t$ and $t+\eps-\tau$. We deduce from assumptions (i) and (ii) that
\begin{align*}
  X_\eps^{(1)}(t)-X_\eps^{(1)}(t+\eps-\tau) \ge v_\eps(\tau-\eps).
\end{align*}
This shows that
\begin{align*}
  |X_\eps^{(1)}(t)&-X_\eps^{(1)}(t+\eps-\tau)|\exp
  \bigg(-\frac{|X_\eps^{(1)}(t)-X_\eps^{(1)}(+\eps-\tau)|^2}{4\tau}
  \bigg) \\
  &\le |X_\eps^{(1)}(t)-X_\eps^{(1)}(0)|
  \exp\bigg(-\frac{v_\eps^2(\tau-\eps)^2}{4\tau}\bigg).
\end{align*}
We infer from \eqref{4.X} that $|X_\eps^{(1)}(t)-X_\eps^{(1)}(0)|$ grows at most linearly, which allows us to apply the dominated convergence theorem to conclude that in the limit $t\to\infty$,
\begin{align*}
  \na c_\eps(x,t) \to -\frac{a}{8\pi}\int_\eps^{t+\eps}\frac{1}{\tau^2}
  v_\eps(\tau-\eps)\exp\bigg(-\frac{v_\eps^2(\tau-\eps)^2}{4\tau}
  \bigg)\dd\tau.
\end{align*}
In the limit $t\to\infty$, the first component of the first equation of \eqref{4.eq} becomes
\begin{align*}
  v_\eps = -\frac{a}{8\pi}\int_\eps^{\infty}\frac{1}{\tau^2}
  v_\eps(\tau-\eps)\exp\bigg(-\frac{v_\eps^2(\tau-\eps)^2}{4\tau}
  \bigg)\dd\tau + u
\end{align*}
or equivalently, 
\begin{align*}
  v_\eps = u\bigg\{1 + \frac{a}{8\pi}\int_\eps^\infty
  \frac{\tau-\eps}{\tau^2}\exp\bigg(-\frac{v_\eps^2(\tau-\eps)^2}{4\tau}
  \bigg)\dd\tau
  \bigg\}^{-1}.
\end{align*}
By Assumption (iii), $v_\eps$ remains bounded for $\eps\to 0$. Moreover, the integrand behaves like $1/\tau$, which means that the integral diverges to infinity. Therefore, $v_\eps\to 0$ as $\eps\to 0$, finishing the proof.
\end{proof}

\section*{Conflict of interest statement}

There is no conflict of interest.


\end{document}